\documentclass{amsart}
\usepackage{amsfonts}

\newtheorem{theorem}{Theorem}
\theoremstyle{plain}
\newtheorem{acknowledgement}{Acknowledgement}

\newtheorem{definition}{Definition}

\newtheorem{lemma}{Lemma}

\newtheorem{remark}{Remark}

\numberwithin{equation}{section}
\input{tcilatex}

\begin{document}
\title[SOME CONVERGENCE AND STABILITY RESULTS FOR TWO KIRK TYPE FIXED POINT
ITERATIVE ALGORITHMS]{SOME CONVERGENCE AND STABILITY RESULTS FOR THE KIRK
MULTISTEP AND KIRK-SP FIXED POINT ITERATIVE ALGORITHMS FOR CONTRACTIVE-LIKE
OPERATORS IN NORMED LINEAR SPACES}
\author{FA\.{I}K G\"{U}RSOY}
\address{Department of Mathematics, Yildiz Technical University, Davutpasa
Campus, Esenler, 34220 Istanbul, Turkey}
\email{faikgursoy02@hotmail.com;fgursoy@yildiz.edu.tr}
\urladdr{http://www.yarbis.yildiz.edu.tr/fgursoy}
\author{VATAN KARAKAYA}
\address{Department of Mathematical Engineering, Yildiz Technical
University, Davutpasa Campus, Esenler, 34210 \.{I}stanbul, Turkey}
\email{vkkaya@yildiz.edu.tr;vkkaya@yahoo.com}
\urladdr{http://www.yarbis.yildiz.edu.tr/vkkaya}
\author{B. E. RHOADES}
\address{Department of Mathematics, Indiana University, Bloomington, IN
47405-7106, USA}
\email{rhoades@indiana.edu}
\urladdr{http://www.math.indiana.edu/people/profile.phtml?id=rhoades}
\date{April 2013}

\begin{abstract}
The purpose of this paper is to introduce a new Kirk type iterative
algorithm called Kirk multistep iteration and to study its convergence. We
also prove some theorems related with the stability results for the
Kirk-multistep and Kirk-SP iterative processes by employing certain
contractive-like operators. Our results generalize and unify some other
results in the literature.
\end{abstract}

\maketitle

\section{Introduction and Preliminaries}

This article is organized as follows. Section 1 outlines some known
contractive mappings and iterative schemes and collects some preliminaries
that will be used in the proofs of our main results. We then propose a new
Kirk type iterative process called Kirk multistep iteration. Section 2
presents a result dealing with the convergence of this new iterative
procedure, which is unifies and extends some other iterative schemes in the
existing literature. Also we prove some theorems related to the stability of
the Kirk multistep and Kirk-SP iterative processes by employing certain
contractive-like operators.

Fixed point iterations are commonly used to solve nonlinear equations
arising in physical systems. Such equations can be transform into a fixed
point equation $Tx=x$ which is solved by some iterative processes of form $%
x_{n+1}=f\left( T,x_{n}\right) $, $n=0,1,2,\ldots $, that converges to a
fixed point of $T$. This is a reason, among a number of reasons, why there
is presently a great deal of interest in the introduction and development of
various iterative algorithms. Consequently iteration schemes abound in the
literature of fixed point theory, for which fixed points of operators have
been approximated over the years by various authors, e.g., \cite{Agarwal,
Das, New, Ishikawa, Mann, Noor, SP, RS7, Takahashi, Thianwan}.

As a background to our exposition, we describe some iteration schemes and
contractive type mappings.

Throughout this paper $%
%TCIMACRO{\U{2115} }%
%BeginExpansion
\mathbb{N}
%EndExpansion
$ denotes the nonnegative integers, including zero. Let $\left\{ \alpha
_{n}\right\} _{n=0}^{\infty }$, $\left\{ \beta _{n}\right\} _{n=0}^{\infty }$%
,$\left\{ \gamma _{n}\right\} _{n=0}^{\infty }$ and $\left\{ \beta
_{n}^{i}\right\} _{n=0}^{\infty }$, $i=\overline{1,k-2}$, $k\geq 2$ be real
sequences in $\left[ 0,1\right) $ satisfying certain conditions.

Rhoades and \c{S}oltuz \cite{RS7}, introduced a multistep iterative
algorithm by%
\begin{equation}
\left\{ 
\begin{array}{c}
x_{0}\in E\text{, \ \ \ \ \ \ \ \ \ \ \ \ \ \ \ \ \ \ \ \ \ \ \ \ \ \ \ \ \
\ \ \ \ \ \ \ \ \ \ \ \ \ \ \ \ } \\ 
x_{n+1}=\left( 1-\alpha _{n}\right) x_{n}+\alpha _{n}Ty_{n}^{1}\text{, \ \ \
\ \ \ \ \ \ \ \ \ \ \ \ \ \ } \\ 
y_{n}^{i}=\left( 1-\beta _{n}^{i}\right) x_{n}+\beta _{n}^{i}Ty_{n}^{i+1}%
\text{, \ \ \ \ \ \ \ \ \ \ \ } \\ 
y_{n}^{k-1}=\left( 1-\beta _{n}^{k-1}\right) x_{n}+\beta
_{n}^{k-1}Tx_{n},~n\in 
%TCIMACRO{\U{2115} }%
%BeginExpansion
\mathbb{N}
%EndExpansion
\text{.}%
\end{array}%
\right.  \label{eqn1}
\end{equation}%
The following multistep iteration was employed in \cite{New, FVB}%
\begin{equation}
\left\{ 
\begin{array}{c}
x_{0}\in E\text{, \ \ \ \ \ \ \ \ \ \ \ \ \ \ \ \ \ \ \ \ \ \ \ \ \ \ \ \ \
\ \ \ \ \ \ \ \ \ \ \ \ \ \ \ \ } \\ 
x_{n+1}=\left( 1-\alpha _{n}\right) y_{n}^{1}+\alpha _{n}Ty_{n}^{1}\text{, \
\ \ \ \ \ \ \ \ \ \ \ \ \ \ \ \ } \\ 
y_{n}^{i}=\left( 1-\beta _{n}^{i}\right) y_{n}^{i+1}+\beta
_{n}^{i}Ty_{n}^{i+1}\text{, \ \ \ \ \ \ \ \ } \\ 
y_{n}^{k-1}=\left( 1-\beta _{n}^{k-1}\right) x_{n}+\beta
_{n}^{k-1}Tx_{n},~n\in 
%TCIMACRO{\U{2115} }%
%BeginExpansion
\mathbb{N}
%EndExpansion
\text{.}%
\end{array}%
\right.  \label{eqn2}
\end{equation}%
By taking $k=3$ and $k=2$ in (1.1) we obtain the well-known Noor \cite{Noor}
and Ishikawa \cite{Ishikawa} iterative schemes, respectively. SP iteration 
\cite{SP} and a new two-step iteration \cite{Thianwan} processes are
obtained by taking $k=3$ and $k=2$ in (1.2), respectively. Both in (1.1) and
in (1.2), if we take $k=2$ with $\beta _{n}^{1}=0$ and $k=2$ with $\beta
_{n}^{1}\equiv 0$, $\alpha _{n}\equiv \lambda $ (const.), then we get the
iterative procedures introduced in \cite{Mann} and \cite{Krasnoselskij},
which are commonly known as the Mann and Krasnoselskij iterations,
respectively. The Krasnoselskij iteration reduces to the Picard iteration 
\cite{Picard} for $\lambda =1$.

The Kirk -SP iterative scheme \cite{Hussain et al} is defined by%
\begin{equation}
\left\{ 
\begin{array}{c}
x_{n+1}=\sum_{i_{1}=0}^{s_{1}}\alpha _{n,i_{1}}T^{i_{1}}y_{n}^{1}\text{, \ \ 
}\sum_{i_{1}=0}^{s_{1}}\alpha _{n,i_{1}}=1\text{, \ \ \ \ \ \ \ \ \ \ \ \ \
\ \ \ \ \ \ \ \ \ } \\ 
y_{n}^{1}=\sum_{i_{2}=0}^{s_{2}}\beta _{n,i_{2}}^{1}T^{i_{2}}y_{n}^{2}\text{%
, \ \ }\sum_{i_{2}=0}^{s_{2}}\beta _{n,i_{2}}^{1}=1\text{, \ \ \ \ \ \ \ \ \
\ \ \ \ \ \ \ \ \ \ } \\ 
y_{n}^{2}=\sum_{i_{3}=0}^{s_{3}}\beta _{n,i_{3}}^{2}T^{i_{3}}x_{n}\text{, \
\ }\sum_{i_{3}=0}^{s_{3}}\beta _{n,i_{3}}^{2}=1\text{, \ \ }\forall n\in 
%TCIMACRO{\U{2115} }%
%BeginExpansion
\mathbb{N}
%EndExpansion
\text{, \ \ \ \ \ \ }%
\end{array}%
\right.  \label{eqn3}
\end{equation}%
where $s_{1}$, $s_{2}$, and $s_{3}$ are fixed integers with $s_{1}\geq
s_{2}\geq s_{3}$ and $\alpha _{n,i_{1}}$, $\beta _{n,i_{2}}^{1}$, $\beta
_{n,i_{3}}^{2}$ are sequences in $\left[ 0,1\right] $ satisfying $\alpha
_{n,i_{1}}\geq 0$, $\alpha _{n,0}\neq 0$, $\beta _{n,i_{2}}^{1}\geq 0$, $%
\beta _{n,0}^{1}\neq 0$,$\ \beta _{n,i_{3}}^{2}\geq 0$, $\beta
_{n,0}^{2}\neq 0$.

Let $X$ be an arbitrary Banach space and $T:X\rightarrow X$ be mapping.

We shall introduce and employ the following iterative scheme, which is
called a Kirk-multistep iteration:%
\begin{equation}
\left\{ 
\begin{array}{c}
x_{0}\in X\text{, \ \ \ \ \ \ \ \ \ \ \ \ \ \ \ \ \ \ \ \ \ \ \ \ \ \ \ \ \
\ \ \ \ \ \ \ \ \ \ \ \ \ \ \ \ \ \ \ \ \ \ \ \ \ \ \ \ \ \ \ \ \ \ } \\ 
x_{n+1}=\alpha _{n,0}x_{n}+\sum\limits_{i_{1}=1}^{s_{1}}\alpha
_{n,i_{1}}T^{i_{1}}y_{n}^{1}\text{, \ \ \ \ \ \ \ \ \ \ \ \ \ \ \ \ \ \ \ \
\ \ \ \ \ \ \ \ \ \ \ } \\ 
y_{n}^{p}=\beta _{n,0}^{p}x_{n}+\sum\limits_{i_{p+1}=1}^{s_{p+1}}\beta
_{n,i_{p+1}}^{p}T^{i_{p+1}}y_{n}^{p+1}\text{, \ }p=\overline{1,k-2}\text{,}
\\ 
y_{n}^{k-1}=\sum\limits_{i_{k}=0}^{s_{k}}\beta _{n,i_{k}}^{k-1}T^{i_{k}}x_{n}%
\text{, \ }k\geq 2\text{, }\forall n\in 
%TCIMACRO{\U{2115} }%
%BeginExpansion
\mathbb{N}
%EndExpansion
\text{, \ \ \ \ \ \ \ \ \ \ \ \ \ \ \ \ \ \ \ \ \ }%
\end{array}%
\right.  \label{eqn4}
\end{equation}%
where $\sum\limits_{i_{1}=0}^{s_{1}}\alpha _{n,i_{1}}=1$, $%
\sum\limits_{i_{p+1}=0}^{s_{p+1}}\beta _{n,i_{p+1}}^{p}=1$ for $p=\overline{%
1,k-1}$; $\alpha _{n,i_{1}}$, $\beta _{n,i_{p+1}}^{p}$ are sequences in $%
\left[ 0,1\right] $ satisfying $\alpha _{n,i_{1}}\geq 0$, $\alpha _{n,0}\neq
0$, $\beta _{n,i_{p+1}}^{p}\geq 0$, $\beta _{n,0}^{p}\neq 0$ for $p=%
\overline{1,k-1}$ and $s_{1}$, $s_{p+1}$ for $p=\overline{1,k-1}$ are fixed
integers with $s_{1}\geq s_{2}\geq \cdots \geq s_{k}$.

By taking $k=3$, $k=2$ and $k=2$ with $s_{2}=0$ in (1.4) we obtain the
Kirk-Noor \cite{CR}, the Kirk-Ishikawa \cite{Olatinwo} and the Kirk-Mann 
\cite{Olatinwo} iterative schemes, respectively. Also,(1.4) gives the usual
Kirk iterative process \cite{Kirk} for $k=2$, with $s_{2}=0$ and $\alpha
_{n,i_{1}}=\alpha _{i_{1}}$. If we put $s_{1}=1$ and $s_{p+1}=1$, $p=%
\overline{1,k-1}$ in (1.4), then we have the usual multistep iteration (1.1)
with $\sum_{i_{1}=0}^{1}\alpha _{n,i_{1}}=1$, $\alpha _{n,1}=\alpha _{n}$, $%
\sum_{i_{p+1}=0}^{1}\beta _{n,i_{p+1}}^{p}=1$, $\beta _{n,1}^{p}=\beta
_{n}^{p}$, $p=\overline{1,k-1}$. The Noor iteration \cite{Noor},\ the
Ishikawa iteration \cite{Ishikawa}, the Mann iteration \cite{Mann}, the
Krasnoselskij iteration \cite{Krasnoselskij} and the Picard iteration \cite%
{Picard} schemes are special cases of the multistep iterative scheme (1.1),
as explained above. So, we conclude that these are special cases of the
Kirk-multistep iterative scheme (1.4).

A particular fixed point iteration generates a theoretical sequence $\left\{
x_{n}\right\} _{n=0}^{\infty }$. In applications, various errors (for
example round-off or discretization of the function $T$ etc.) occur during
computation of the sequence $\left\{ x_{n}\right\} _{n=0}^{\infty }$.
Because of these errors we cannot obtain the theoretical sequence $\left\{
x_{n}\right\} _{n=0}^{\infty }$, but an approximate sequence $\left\{
y_{n}\right\} _{n=0}^{\infty }$ instead. We shall say that the iterative
process is $T$-stable or stable with respect to $T$ if and only if $\left\{
x_{n}\right\} _{n=0}^{\infty }$ converges to a fixed point $q$ of $T$, then $%
\left\{ y_{n}\right\} _{n=0}^{\infty }$ converges to $q=Tq$.

The initiator of this kind study is M. Urabe \cite{Urabe} while a formal
definition for the stability of general iterative schemes is given by Harder
and Hicks \cite{HarHic1, HarHic2} as follows:

\begin{definition}
Let $\left( X,d\right) $ be a complete metric space, $T$ a self map of $X$.
Suppose that $F_{T}=\left\{ q\in X:q=Tq\right\} $ is the set of fixed points
of $T$. Let $\left\{ x_{n}\right\} _{n=0}^{\infty }\subset X$ be a sequence
generated by an iterative process defined by 
\begin{equation}
x_{n+1}=f\left( T,x_{n}\right) ,n=0,1,\ldots \text{,}  \label{eqn5}
\end{equation}%
where $x_{0}\in X$ is the initial approximation and $f$ is some function.
Let $\left\{ y_{n}\right\} _{n=0}^{\infty }\subset X$ be an arbitrary
sequence and set $\varepsilon _{n}=d\left( y_{n+1},f\left( T,y_{n}\right)
\right) $, $n=0,1,\ldots $. Then, the iterative process (1.5) is said to be
\ $T$-stable or stable with respect to $T$ if and only if $%
\lim_{n\rightarrow \infty }\varepsilon _{n}=0\Rightarrow \lim_{n\rightarrow
\infty }y_{n}=q$.
\end{definition}

In the last three decades, a large literature has developed dealing with the
stability of various well-known iterative schemes for different classes of
operators. Several authors who have made contributions to the study of
stability of fixed point iterative procedures are Ostrowski \cite{Ostrowski}%
, Harder \cite{Harder}, Harder and Hicks \cite{HarHic1, HarHic2}, Rhoades 
\cite{Rhds3, Rhds4}, Berinde \cite{Berinde1, Berinde2}, Osilike \cite%
{Osilike1, Osilike2}, Osilike and Udomene \cite{Osilike}, Olatinwo \cite%
{Olatinwo, Olatinwo1, Olatinwo2}, Chugh and Kumar \cite{CR}, and several
references contained therein.

A pioneering result on the stability of iterative procedures established in
metric space and normed linear space for the Picard iteration is due to
Ostrowski \cite{Ostrowski}, which states that: \textit{Let }$(X,d)$\textit{\
be a complete metric space and }$T:X\rightarrow X$\textit{\ a Banach
contraction mapping, i.e., }%
\begin{equation}
d(Tx,Ty)\leq \lambda d(x,y)\text{ \textit{for all} }x,y\in X\text{,}
\label{eqn6}
\end{equation}%
\textit{where }$\lambda \in \left[ 0,1\right) $\textit{. Let }$q\in X$%
\textit{\ be the fixed point of }$T$\textit{, }$x_{0}\in X$\textit{\ and }$%
x_{n+1}=Tx_{n}$\textit{, }$n=0,1,2,...$\textit{. Suppose that }$%
\{y_{n}\}_{n=0}^{\infty }$\textit{\ is a sequence in }$X$\textit{\ and }$%
\varepsilon _{n}=d\left( y_{n+1},Ty_{n}\right) $\textit{. Then}%
\begin{equation}
d\left( q,y_{n+1}\right) \leq d\left( q,x_{n+1}\right) +\lambda
^{n+1}d\left( x_{0},y_{0}\right) +\sum\limits_{i=0}^{n}\lambda
^{n-r}\varepsilon _{i}\text{.}  \label{eqn7}
\end{equation}%
\textit{Moreover, }$\lim_{n\rightarrow \infty }y_{n}=q\Leftrightarrow
\lim_{n\rightarrow \infty }\varepsilon _{n}=0$\textit{.}

Using Definition 1, Harder and Hicks \cite{HarHic1, HarHic2} proved some
stability theorems for well-known Picard, Mann and Kirk's iterations by
employing several classes of contractive type operators. Rhoades \cite%
{Rhds3, Rhds4} extended the results of Harder and Hicks \cite{HarHic2} by
utilizing the following two different classes of contractive operators of
Ciric's type, respectively: there exists a $\lambda \in \left[ 0,1\right) $
such that for each pair $x,y\in X$%
\begin{equation}
d\left( Tx,Ty\right) \leq \lambda \max \left\{ d\left( x,y\right) ,d\left(
x,Ty\right) ,d\left( y,Tx\right) \right\} \text{,}  \label{eqn8}
\end{equation}%
and%
\begin{equation}
d\left( Tx,Ty\right) \leq \lambda \max \left\{ d\left( x,y\right) ,\left\{
d\left( x,Tx\right) +d\left( y,Ty\right) \right\} /2,d\left( x,Ty\right)
,d\left( y,Tx\right) \right\} \text{.}  \label{eqn9}
\end{equation}%
Later Osilike \cite{Osilike1} further generalized and extended some of the
results in \cite{Rhds3} by using a large class of contractive type operators 
$T$ satisfying the following condition, which is more general than those of
Rhoades \cite{Rhds3, Rhds4} and Harder and Hicks \cite{HarHic2}:%
\begin{equation}
d(Tx,Ty)\leq Ld(x,Tx)+\lambda d(x,y)\text{,}  \label{eqn10}
\end{equation}%
for some $\lambda \in \lbrack 0,1)$, $L\geq 0$ and for all $x,y\in X$.

By employing the contractive condition (1.10), Osilike and Udomene proved
some stability results for the Picard, Ishikawa and Kirk's iteration in \cite%
{Osilike} where a new and shorter method than those mentioned above was
used. Using the same method of proof as in \cite{Osilike}, Berinde \cite%
{Berinde2} again established the stability results in Harder and Hicks \cite%
{HarHic2}.

In \cite{Imoru}, Imoru and Olatinwo extended some of the stability results
of \cite{Berinde2, HarHic2, Osilike, Osilike1, Rhds3, Rhds4} by employing a
much more general class of operators $T$ satisfying the following
contractive condition:%
\begin{equation}
d(Tx,Ty)\leq \varphi (d(x,Tx))+\lambda d(x,y)\text{, }\forall x\text{, }y\in
X\text{,}  \label{eqn11}
\end{equation}%
where $\lambda \in \lbrack 0,1)$ and $\varphi :%
%TCIMACRO{\U{211d} }%
%BeginExpansion
\mathbb{R}
%EndExpansion
^{+}\rightarrow 
%TCIMACRO{\U{211d} }%
%BeginExpansion
\mathbb{R}
%EndExpansion
^{+}$ is a monotone increasing function with $\varphi (0)=0$.

\begin{remark}
\cite{New, FVB} A map satisfying (1.11) need not have a fixed point.
However, using (1.11), it is obvious that if T has a fixed point, then it is
unique.
\end{remark}

Continuing the above mentioned trend, Olatinwo \cite{Olatinwo} studied the
stability of the Kirk-Mann and Kirk-Ishikawa iterative processes by
utilizing contractive condition (1.11). The results of \cite{Olatinwo} are
generalizations of some of the results of \cite{Berinde2, HarHic2,
Olantiwo3, Olantiwo4, Olantiwo5, Osilike, Osilike1, Rhds3, Rhds4}.

Recently Chugh and Kumar \cite{CR} improved and extended the results of \cite%
{Olatinwo}, and some of the references cited therein, by introducing the
Kirk-Noor iterative algorithm.

We end this section with some lemmas which will be useful in proving our
main results.

\begin{lemma}
\cite{Berinde3} If $\sigma $ is a real number such that $\sigma \in \left[
0,1\right) $, and $\left\{ \varepsilon _{n}\right\} _{n=0}^{\infty }$ is a
sequence of nonegative numbers such that $\lim_{n\rightarrow \infty
}\varepsilon _{n}=0$, then, for any sequence of positive numbers $\left\{
u_{n}\right\} _{n=0}^{\infty }$ satisfying%
\begin{equation}
u_{n+1}\leq \sigma u_{n}+\varepsilon _{n}\text{, }\forall n\in 
%TCIMACRO{\U{2115} }%
%BeginExpansion
\mathbb{N}
%EndExpansion
\text{,}  \label{eqn12}
\end{equation}%
we have $\lim_{n\rightarrow \infty }u_{n}=0$.
\end{lemma}

\begin{lemma}
\cite{Olatinwo} Let $(X$, $\left\Vert \cdot \right\Vert )$ be a normed
linear space and let $T$ be a selfmap of $X$ satisfying (1.11). Let $\varphi
:%
%TCIMACRO{\U{211d} }%
%BeginExpansion
\mathbb{R}
%EndExpansion
^{+}\rightarrow 
%TCIMACRO{\U{211d} }%
%BeginExpansion
\mathbb{R}
%EndExpansion
^{+}$ be a subadditive, monotone increasing function such that $\varphi
(0)=0 $, $\varphi (Lu)\leq L\varphi (u)$, $L\geq 0$, $u\in 
%TCIMACRO{\U{211d} }%
%BeginExpansion
\mathbb{R}
%EndExpansion
^{+}$. Then, $\forall i\in 
%TCIMACRO{\U{2115} }%
%BeginExpansion
\mathbb{N}
%EndExpansion
$, $L\geq 0$ and $\forall x$, $y\in X$%
\begin{equation}
\left\Vert T^{i}x-T^{i}y\right\Vert \leq \sum\limits_{j=1}^{i}\binom{i}{j}%
a^{i-j}\varphi ^{j}\left( \left\Vert x-Tx\right\Vert \right)
+a^{i}\left\Vert x-y\right\Vert \text{.}  \label{eqn13}
\end{equation}
\end{lemma}

\begin{remark}
Note that $a\in \left[ 0,1\right) $ in the equation (1.13).
\end{remark}

\section{Main Results}

For simplicity we assume in the following three theorems that $X$ is a
normed linear space, $T$ is a self map of $X$ satisfying the contractive
condition (1.11) with $F_{T}\neq \emptyset $, and $\varphi :%
%TCIMACRO{\U{211d} }%
%BeginExpansion
\mathbb{R}
%EndExpansion
^{+}\rightarrow 
%TCIMACRO{\U{211d} }%
%BeginExpansion
\mathbb{R}
%EndExpansion
^{+}$ is a subadditive monotone increasing function such that $\varphi
\left( 0\right) =0$ and $\varphi \left( Lu\right) \leq L\varphi \left(
u\right) $, $L\geq 0$, $u\in 
%TCIMACRO{\U{211d} }%
%BeginExpansion
\mathbb{R}
%EndExpansion
^{+}$.

\begin{theorem}
Let $\left\{ x_{n}\right\} _{n\in 
%TCIMACRO{\U{2115} }%
%BeginExpansion
\mathbb{N}
%EndExpansion
}$ be a sequence generated by the Kirk-multistep iterative scheme (1.4).
Suppose that $T$ has a fixed point $q$. Then the iterative sequence $\left\{
x_{n}\right\} _{n\in 
%TCIMACRO{\U{2115} }%
%BeginExpansion
\mathbb{N}
%EndExpansion
}$ converges strongly to $q$.
\end{theorem}

\begin{proof}
The uniqueness of $q$ follows from (1.13). We shall now prove that $%
x_{n}\rightarrow q$.

Using (1.4) and Lemma 2, we get%
\begin{eqnarray}
\left\Vert x_{n+1}-q\right\Vert &=&\left\Vert \alpha _{n,0}x_{n}-\alpha
_{n,0}q+\sum\limits_{i_{1}=1}^{s_{1}}\alpha
_{n,i_{1}}T^{i_{1}}y_{n}^{1}-\sum\limits_{i_{1}=1}^{s_{1}}\alpha
_{n,i_{1}}q\right\Vert  \notag \\
&\leq &\alpha _{n,0}\left\Vert x_{n}-q\right\Vert
+\sum\limits_{i_{1}=1}^{s_{1}}\alpha _{n,i_{1}}\left\Vert
T^{i_{1}}y_{n}^{1}-T^{i_{1}}q\right\Vert  \notag \\
&\leq &\alpha _{n,0}\left\Vert x_{n}-q\right\Vert  \notag \\
&&+\sum\limits_{i_{1}=1}^{s_{1}}\alpha _{n,i_{1}}\left\{
\sum\limits_{j=1}^{i_{1}}\binom{i_{1}}{j}a^{i_{1}-j}\varphi ^{j}\left(
\left\Vert q-Tq\right\Vert \right) +a^{i_{1}}\left\Vert
y_{n}^{1}-q\right\Vert \right\}  \notag \\
&=&\alpha _{n,0}\left\Vert x_{n}-q\right\Vert +\left(
\sum\limits_{i_{1}=1}^{s_{1}}\alpha _{n,i_{1}}a^{i_{1}}\right) \left\Vert
y_{n}^{1}-q\right\Vert \text{,}  \label{eqn14}
\end{eqnarray}%
\begin{eqnarray}
\left\Vert y_{n}^{1}-q\right\Vert &=&\left\Vert \beta
_{n,0}^{1}x_{n}+\sum\limits_{i_{2}=1}^{s_{2}}\beta
_{n,i_{2}}^{1}T^{i_{2}}y_{n}^{2}-\sum\limits_{i_{2}=0}^{s_{2}}\beta
_{n,i_{2}}^{1}q\right\Vert  \notag \\
&\leq &\beta _{n,0}^{1}\left\Vert x_{n}-q\right\Vert
+\sum\limits_{i_{2}=1}^{s_{2}}\beta _{n,i_{2}}^{1}\left\Vert
T^{i_{2}}y_{n}^{2}-T^{i_{2}}q\right\Vert  \notag \\
&\leq &\beta _{n,0}^{1}\left\Vert x_{n}-q\right\Vert  \notag \\
&&+\sum\limits_{i_{2}=1}^{s_{2}}\beta _{n,i_{2}}^{1}\left\{
\sum\limits_{j=1}^{i_{2}}\binom{i_{2}}{j}a^{i_{2}-j}\varphi ^{j}\left(
\left\Vert q-Tq\right\Vert \right) +a^{i_{2}}\left\Vert
y_{n}^{2}-q\right\Vert \right\}  \notag \\
&=&\beta _{n,0}^{1}\left\Vert x_{n}-q\right\Vert +\left(
\sum\limits_{i_{2}=1}^{s_{2}}\beta _{n,i_{2}}^{1}a^{i_{2}}\right) \left\Vert
y_{n}^{2}-q\right\Vert \text{,}  \label{eqn15}
\end{eqnarray}%
\begin{eqnarray}
\left\Vert y_{n}^{2}-q\right\Vert &=&\left\Vert \beta
_{n,0}^{2}x_{n}+\sum\limits_{i_{3}=1}^{s_{3}}\beta
_{n,i_{3}}^{2}T^{i_{3}}y_{n}^{3}-q\right\Vert  \notag \\
&\leq &\beta _{n,0}^{2}\left\Vert x_{n}-q\right\Vert
+\sum\limits_{i_{3}=1}^{s_{3}}\beta _{n,i_{3}}^{2}\left\Vert
T^{i_{3}}y_{n}^{3}-T^{i_{3}}q\right\Vert  \notag \\
&\leq &\beta _{n,0}^{2}\left\Vert x_{n}-q\right\Vert  \notag \\
&&+\sum\limits_{i_{3}=1}^{s_{3}}\beta _{n,i_{3}}^{2}\left\{
\sum\limits_{j=1}^{i_{3}}\binom{i_{3}}{j}a^{i_{3}-j}\varphi ^{j}\left(
\left\Vert q-Tq\right\Vert \right) +a^{i_{3}}\left\Vert
y_{n}^{3}-q\right\Vert \right\}  \notag \\
&=&\beta _{n,0}^{2}\left\Vert x_{n}-q\right\Vert +\left(
\sum\limits_{i_{3}=1}^{s_{3}}\beta _{n,i_{3}}^{2}a^{i_{3}}\right) \left\Vert
y_{n}^{3}-q\right\Vert \text{,}  \label{eqn16}
\end{eqnarray}%
\begin{eqnarray}
\left\Vert y_{n}^{3}-q\right\Vert &=&\left\Vert \beta
_{n,0}^{3}x_{n}+\sum\limits_{i_{4}=1}^{s_{4}}\beta
_{n,i_{4}}^{3}T^{i_{4}}y_{n}^{4}-q\right\Vert  \notag \\
&\leq &\beta _{n,0}^{3}\left\Vert x_{n}-q\right\Vert
+\sum\limits_{i_{4}=1}^{s_{4}}\beta _{n,i_{4}}^{3}\left\Vert
T^{i_{4}}y_{n}^{4}-T^{i_{4}}q\right\Vert  \notag \\
&\leq &\beta _{n,0}^{3}\left\Vert x_{n}-q\right\Vert  \notag \\
&&+\sum\limits_{i_{4}=1}^{s_{4}}\beta _{n,i_{4}}^{3}\left\{
\sum\limits_{j=1}^{i_{4}}\binom{i_{4}}{j}a^{i_{4}-j}\varphi ^{j}\left(
\left\Vert q-Tq\right\Vert \right) +a^{i_{4}}\left\Vert
y_{n}^{4}-q\right\Vert \right\}  \notag \\
&=&\beta _{n,0}^{3}\left\Vert x_{n}-q\right\Vert +\left(
\sum\limits_{i_{4}=1}^{s_{4}}\beta _{n,i_{4}}^{3}a^{i_{4}}\right) \left\Vert
y_{n}^{4}-q\right\Vert \text{.}  \label{eqn17}
\end{eqnarray}%
By combining (2.1), (2.2), (2.3), and (2.4) we obtain%
\begin{eqnarray}
\left\Vert x_{n+1}-q\right\Vert &\leq &\left\{ \alpha _{n,0}+\left(
\sum\limits_{i_{1}=1}^{s_{1}}\alpha _{n,i_{1}}a^{i_{1}}\right) \beta
_{n,0}^{1}+\left( \sum\limits_{i_{1}=1}^{s_{1}}\alpha
_{n,i_{1}}a^{i_{1}}\right) \left( \sum\limits_{i_{2}=1}^{s_{2}}\beta
_{n,i_{2}}^{1}a^{i_{2}}\right) \beta _{n,0}^{2}\right.  \notag \\
&&\left. +\left( \sum\limits_{i_{1}=1}^{s_{1}}\alpha
_{n,i_{1}}a^{i_{1}}\right) \left( \sum\limits_{i_{2}=1}^{s_{2}}\beta
_{n,i_{2}}^{1}a^{i_{2}}\right) \left( \sum\limits_{i_{3}=1}^{s_{3}}\beta
_{n,i_{3}}^{2}a^{i_{3}}\right) \beta _{n,0}^{3}\right\} \left\Vert
x_{n}-q\right\Vert  \notag \\
&&+\left( \sum\limits_{i_{1}=1}^{s_{1}}\alpha _{n,i_{1}}a^{i_{1}}\right)
\left( \sum\limits_{i_{2}=1}^{s_{2}}\beta _{n,i_{2}}^{1}a^{i_{2}}\right)
\left( \sum\limits_{i_{3}=1}^{s_{3}}\beta _{n,i_{3}}^{2}a^{i_{3}}\right)
\left( \sum\limits_{i_{4}=1}^{s_{4}}\beta _{n,i_{4}}^{3}a^{i_{4}}\right)
\left\Vert y_{n}^{4}-q\right\Vert \text{.}  \label{eqn18}
\end{eqnarray}%
Continuing the above process we have%
\begin{eqnarray}
\left\Vert x_{n+1}-q\right\Vert &\leq &\left\{ \alpha _{n,0}+\left(
\sum\limits_{i_{1}=1}^{s_{1}}\alpha _{n,i_{1}}a^{i_{1}}\right) \beta
_{n,0}^{1}+\left( \sum\limits_{i_{1}=1}^{s_{1}}\alpha
_{n,i_{1}}a^{i_{1}}\right) \left( \sum\limits_{i_{2}=1}^{s_{2}}\beta
_{n,i_{2}}^{1}a^{i_{2}}\right) \beta _{n,0}^{2}\right.  \notag \\
&&+\cdots  \notag \\
&&+\left. \left( \sum\limits_{i_{1}=1}^{s_{1}}\alpha
_{n,i_{1}}a^{i_{1}}\right) \left( \sum\limits_{i_{2}=1}^{s_{2}}\beta
_{n,i_{2}}^{1}a^{i_{2}}\right) \cdots \left(
\sum\limits_{i_{k-2}=1}^{s_{k-2}}\beta _{n,i_{k-2}}^{k-3}a^{i_{k-2}}\right)
\beta _{n,0}^{k-2}\right\} \left\Vert x_{n}-q\right\Vert  \notag \\
&&+\left( \sum\limits_{i_{1}=1}^{s_{1}}\alpha _{n,i_{1}}a^{i_{1}}\right)
\left( \sum\limits_{i_{2}=1}^{s_{2}}\beta _{n,i_{2}}^{1}a^{i_{2}}\right)
\cdots \left( \sum\limits_{i_{k-1}=1}^{s_{k-1}}\beta
_{n,i_{k-1}}^{k-2}a^{i_{k-1}}\right) \left\Vert y_{n}^{k-1}-q\right\Vert 
\text{.}  \label{eqn19}
\end{eqnarray}%
Using again (1.4) and Lemma 2, we get 
\begin{eqnarray}
\left\Vert y_{n}^{k-1}-q\right\Vert &=&\left\Vert \beta _{n,0}^{k-1}\left(
x_{n}-q\right) +\sum\limits_{i_{k}=1}^{s_{k}}\beta _{n,i_{k}}^{k-1}\left(
T^{i_{k}}x_{n}-T^{i_{k}}q\right) \right\Vert  \notag \\
&\leq &\beta _{n,0}^{k-1}\left\Vert x_{n}-q\right\Vert
+\sum\limits_{i_{k}=1}^{s_{k}}\beta _{n,i_{k}}^{k-1}\left\Vert
T^{i_{k}}x_{n}-T^{i_{k}}q\right\Vert  \notag \\
&\leq &\beta _{n,0}^{k-1}\left\Vert x_{n}-q\right\Vert  \notag \\
&&+\sum\limits_{i_{k}=1}^{s_{k}}\beta _{n,i_{k}}^{k-1}\left\{
\sum\limits_{j=1}^{i_{k}}\binom{i_{k}}{j}a^{i_{k}-j}\varphi ^{j}\left(
\left\Vert q-Tq\right\Vert \right) +a^{i_{k}}\left\Vert x_{n}-q\right\Vert
\right\}  \notag \\
&=&\beta _{n,0}^{k-1}\left\Vert x_{n}-q\right\Vert +\left(
\sum\limits_{i_{k}=1}^{s_{k}}\beta _{n,i_{k}}^{k-1}a^{i_{k}}\right)
\left\Vert x_{n}-q\right\Vert  \notag \\
&=&\left( \sum\limits_{i_{k}=0}^{s_{k}}\beta
_{n,i_{k}}^{k-1}a^{i_{k}}\right) \left\Vert x_{n}-q\right\Vert \text{.}
\label{eqn20}
\end{eqnarray}%
Substituting (2.7) into (2.6) we derive%
\begin{eqnarray}
\left\Vert x_{n+1}-q\right\Vert &\leq &\left\{ \alpha _{n,0}+\left(
\sum\limits_{i_{1}=1}^{s_{1}}\alpha _{n,i_{1}}a^{i_{1}}\right) \beta
_{n,0}^{1}+\left( \sum\limits_{i_{1}=1}^{s_{1}}\alpha
_{n,i_{1}}a^{i_{1}}\right) \left( \sum\limits_{i_{2}=1}^{s_{2}}\beta
_{n,i_{2}}^{1}a^{i_{2}}\right) \beta _{n,0}^{2}\right.  \notag \\
&&+\cdots +\left( \sum\limits_{i_{1}=1}^{s_{1}}\alpha
_{n,i_{1}}a^{i_{1}}\right) \left( \sum\limits_{i_{2}=1}^{s_{2}}\beta
_{n,i_{2}}^{1}a^{i_{2}}\right) \cdots \left(
\sum\limits_{i_{k-2}=1}^{s_{k-2}}\beta _{n,i_{k-2}}^{k-3}a^{i_{k-2}}\right)
\beta _{n,0}^{k-2}  \notag \\
&&\left. +\left( \sum\limits_{i_{1}=1}^{s_{1}}\alpha
_{n,i_{1}}a^{i_{1}}\right) \left( \sum\limits_{i_{2}=1}^{s_{2}}\beta
_{n,i_{2}}^{1}a^{i_{2}}\right) \right.  \notag \\
&&\left. \cdots \left( \sum\limits_{i_{k-1}=1}^{s_{k-1}}\beta
_{n,i_{k-1}}^{k-2}a^{i_{k-1}}\right) \left(
\sum\limits_{i_{k}=0}^{s_{k}}\beta _{n,i_{k}}^{k-1}a^{i_{k}}\right) \right\}
\left\Vert x_{n}-q\right\Vert \text{.}  \label{eqn21}
\end{eqnarray}%
Define%
\begin{eqnarray}
\sigma &:&=\alpha _{n,0}+\left( \sum\limits_{i_{1}=1}^{s_{1}}\alpha
_{n,i_{1}}a^{i_{1}}\right) \beta _{n,0}^{1}+\left(
\sum\limits_{i_{1}=1}^{s_{1}}\alpha _{n,i_{1}}a^{i_{1}}\right) \left(
\sum\limits_{i_{2}=1}^{s_{2}}\beta _{n,i_{2}}^{1}a^{i_{2}}\right) \beta
_{n,0}^{2}  \notag \\
&&+\cdots +\left( \sum\limits_{i_{1}=1}^{s_{1}}\alpha
_{n,i_{1}}a^{i_{1}}\right) \left( \sum\limits_{i_{2}=1}^{s_{2}}\beta
_{n,i_{2}}^{1}a^{i_{2}}\right) \cdots \left(
\sum\limits_{i_{k-2}=1}^{s_{k-2}}\beta _{n,i_{k-2}}^{k-3}a^{i_{k-2}}\right)
\beta _{n,0}^{k-2}  \notag \\
&&+\left( \sum\limits_{i_{1}=1}^{s_{1}}\alpha _{n,i_{1}}a^{i_{1}}\right)
\left( \sum\limits_{i_{2}=1}^{s_{2}}\beta _{n,i_{2}}^{1}a^{i_{2}}\right)
\cdots \left( \sum\limits_{i_{k-1}=1}^{s_{k-1}}\beta
_{n,i_{k-1}}^{k-2}a^{i_{k-1}}\right) \left(
\sum\limits_{i_{k}=0}^{s_{k}}\beta _{n,i_{k}}^{k-1}a^{i_{k}}\right) \text{.}
\label{eqn22}
\end{eqnarray}%
Now we show that $\sigma \in \left[ 0,1\right) $. Since $a^{i_{k}}\in \left[
0,1\right) $, $\alpha _{n,0}>0$, $\sum\limits_{i_{1}=0}^{s_{1}}\alpha
_{n,i_{1}}=1$ and $\sum\limits_{i_{p+1}=0}^{s_{p+1}}\beta _{n,i_{p+1}}^{p}=1$
for $p=\overline{1,k-1}$, we obtain%
\begin{eqnarray}
\sigma &<&\alpha _{n,0}+\left( 1-\alpha _{n,0}\right) \beta
_{n,0}^{1}+\left( 1-\alpha _{n,0}\right) \left( 1-\beta _{n,0}^{1}\right)
\beta _{n,0}^{2}  \notag \\
&&+\cdots +\left( 1-\alpha _{n,0}\right) \left( 1-\beta _{n,0}^{1}\right)
\cdots \left( 1-\beta _{n,0}^{k-3}\right) \beta _{n,0}^{k-2}  \notag \\
&&+\left( 1-\alpha _{n,0}\right) \left( 1-\beta _{n,0}^{1}\right) \cdots
\left( 1-\beta _{n,0}^{k-3}\right) \left( 1-\beta _{n,0}^{k-2}\right)  \notag
\\
&=&1\text{.}  \label{eqn23}
\end{eqnarray}%
By an application of Lemma 1 to (2.8), $\lim_{n\rightarrow \infty }x_{n}=q$.
\end{proof}

\begin{theorem}
Let $x_{0}\in X$ and $\left\{ x_{n}\right\} _{n\in 
%TCIMACRO{\U{2115} }%
%BeginExpansion
\mathbb{N}
%EndExpansion
}$ be a sequence generated by the Kirk-multistep iterative scheme (1.4).
Suppose that $T$ has a fixed point $q$. Then the Kirk multistep iterative
scheme (1.4) is $T$-stable.
\end{theorem}

\begin{proof}
Let $\left\{ y_{n}\right\} _{n\in 
%TCIMACRO{\U{2115} }%
%BeginExpansion
\mathbb{N}
%EndExpansion
}\subset X$, $\left\{ u_{n}^{p}\right\} _{n\in 
%TCIMACRO{\U{2115} }%
%BeginExpansion
\mathbb{N}
%EndExpansion
}$, for $p=\overline{1,k-1}$ be arbitrary sequences in $X$. Let $\varepsilon
_{n}=\left\Vert y_{n+1}-\alpha
_{n,0}y_{n}-\sum\limits_{i_{1}=1}^{s_{1}}\alpha
_{n,i_{1}}T^{i_{1}}u_{n}^{1}\right\Vert $, $n=0,1,2,\ldots ,$ where $%
u_{n}^{p}=\beta _{n,0}^{p}y_{n}+\sum\limits_{i_{p+1}=1}^{s_{p+1}}\beta
_{n,i_{p+1}}^{p}T^{i_{p+1}}u_{n}^{p+1}$,\ $p=\overline{1,k-2}$, $%
u_{n}^{k-1}=\sum\limits_{i_{k}=0}^{s_{k}}\beta
_{n,i_{k}}^{k-1}T^{i_{k}}y_{n} $, $k\geq 2$ and let $\lim_{n\rightarrow
\infty }\varepsilon _{n}=0$. We shall prove that $\lim_{n\rightarrow \infty
}y_{n}=q$.

It follows from (1.4) and Lemma 2 that%
\begin{eqnarray}
\left\Vert y_{n+1}-q\right\Vert &=&\left\Vert y_{n+1}-\alpha
_{n,0}y_{n}-\sum\limits_{i_{1}=1}^{s_{1}}\alpha
_{n,i_{1}}T^{i_{1}}u_{n}^{1}+\alpha
_{n,0}y_{n}+\sum\limits_{i_{1}=1}^{s_{1}}\alpha
_{n,i_{1}}T^{i_{1}}u_{n}^{1}-q\right\Vert  \notag \\
&\leq &\varepsilon _{n}+\left\Vert \alpha _{n,0}\left( y_{n}-q\right)
+\sum\limits_{i_{1}=1}^{s_{1}}\alpha _{n,i_{1}}\left(
T^{i_{1}}u_{n}^{1}-T^{i_{1}}q\right) \right\Vert  \notag \\
&\leq &\alpha _{n,0}\left\Vert y_{n}-q\right\Vert +\varepsilon
_{n}+\sum\limits_{i_{1}=1}^{s_{1}}\alpha _{n,i_{1}}\left\Vert
T^{i_{1}}u_{n}^{1}-T^{i_{1}}q\right\Vert  \notag \\
&\leq &\alpha _{n,0}\left\Vert y_{n}-q\right\Vert +\varepsilon _{n}  \notag
\\
&&+\sum\limits_{i_{1}=1}^{s_{1}}\alpha _{n,i_{1}}\left\{
\sum\limits_{j=1}^{i_{1}}\binom{i_{1}}{j}a^{i_{1}-j}\varphi ^{j}\left(
\left\Vert q-Tq\right\Vert \right) +a^{i_{1}}\left\Vert
u_{n}^{1}-q\right\Vert \right\}  \notag \\
&=&\alpha _{n,0}\left\Vert y_{n}-q\right\Vert +\varepsilon _{n}+\left(
\sum\limits_{i_{1}=1}^{s_{1}}\alpha _{n,i_{1}}a^{i_{1}}\right) \left\Vert
u_{n}^{1}-q\right\Vert \text{,}  \label{eqn24}
\end{eqnarray}%
\begin{eqnarray}
\left\Vert u_{n}^{1}-q\right\Vert &=&\left\Vert \beta _{n,0}^{1}\left(
y_{n}-q\right) +\sum\limits_{i_{2}=1}^{s_{2}}\beta _{n,i_{2}}^{1}\left(
T^{i_{2}}u_{n}^{2}-T^{i_{2}}q\right) \right\Vert  \notag \\
&\leq &\beta _{n,0}^{1}\left\Vert y_{n}-q\right\Vert
+\sum\limits_{i_{2}=1}^{s_{2}}\beta _{n,i_{2}}^{1}\left\Vert
T^{i_{2}}u_{n}^{2}-T^{i_{2}}q\right\Vert  \notag \\
&\leq &\beta _{n,0}^{1}\left\Vert y_{n}-q\right\Vert  \notag \\
&&+\sum\limits_{i_{2}=1}^{s_{2}}\beta _{n,i_{2}}^{1}\left\{
\sum\limits_{j=1}^{i_{2}}\binom{i_{2}}{j}a^{i_{2}-j}\varphi ^{j}\left(
\left\Vert q-Tq\right\Vert \right) +a^{i_{2}}\left\Vert
u_{n}^{2}-q\right\Vert \right\}  \notag \\
&=&\beta _{n,0}^{1}\left\Vert y_{n}-q\right\Vert +\left(
\sum\limits_{i_{2}=1}^{s_{2}}\beta _{n,i_{2}}^{1}a^{i_{2}}\right) \left\Vert
u_{n}^{2}-q\right\Vert \text{,}  \label{eqn25}
\end{eqnarray}%
\begin{eqnarray}
\left\Vert u_{n}^{2}-q\right\Vert &=&\left\Vert \beta _{n,0}^{2}\left(
y_{n}-q\right) +\sum\limits_{i_{3}=1}^{s_{3}}\beta _{n,i_{3}}^{2}\left(
T^{i_{3}}u_{n}^{3}-T^{i_{3}}q\right) \right\Vert  \notag \\
&\leq &\beta _{n,0}^{2}\left\Vert y_{n}-q\right\Vert
+\sum\limits_{i_{3}=1}^{s_{3}}\beta _{n,i_{3}}^{2}\left\Vert
T^{i_{3}}u_{n}^{3}-T^{i_{3}}q\right\Vert  \notag \\
&\leq &\beta _{n,0}^{2}\left\Vert y_{n}-q\right\Vert  \notag \\
&&+\sum\limits_{i_{3}=1}^{s_{3}}\beta _{n,i_{3}}^{2}\left\{
\sum\limits_{j=1}^{i_{3}}\binom{i_{3}}{j}a^{i_{3}-j}\varphi ^{j}\left(
\left\Vert q-Tq\right\Vert \right) +a^{i_{3}}\left\Vert
u_{n}^{3}-q\right\Vert \right\}  \notag \\
&=&\beta _{n,0}^{2}\left\Vert y_{n}-q\right\Vert +\left(
\sum\limits_{i_{3}=1}^{s_{3}}\beta _{n,i_{3}}^{2}a^{i_{3}}\right) \left\Vert
u_{n}^{3}-q\right\Vert \text{.}  \label{eqn26}
\end{eqnarray}%
Combining (2.11), (2.12), and (2.13) we have%
\begin{eqnarray}
\left\Vert y_{n+1}-q\right\Vert &\leq &\left(
\sum\limits_{i_{1}=1}^{s_{1}}\alpha _{n,i_{1}}a^{i_{1}}\right) \left(
\sum\limits_{i_{2}=1}^{s_{2}}\beta _{n,i_{2}}^{1}a^{i_{2}}\right) \left(
\sum\limits_{i_{3}=1}^{s_{3}}\beta _{n,i_{3}}^{2}a^{i_{3}}\right) \left\Vert
u_{n}^{3}-q\right\Vert  \notag \\
&&+\left( \sum\limits_{i_{1}=1}^{s_{1}}\alpha _{n,i_{1}}a^{i_{1}}\right)
\left( \sum\limits_{i_{2}=1}^{s_{2}}\beta _{n,i_{2}}^{1}a^{i_{2}}\right)
\beta _{n,0}^{2}\left\Vert y_{n}-q\right\Vert  \notag \\
&&+\left( \sum\limits_{i_{1}=1}^{s_{1}}\alpha _{n,i_{1}}a^{i_{1}}\right)
\beta _{n,0}^{1}\left\Vert y_{n}-q\right\Vert +\alpha _{n,0}\left\Vert
y_{n}-q\right\Vert +\varepsilon _{n}\text{.}  \label{eqn27}
\end{eqnarray}%
By induction we get%
\begin{eqnarray}
\left\Vert y_{n+1}-q\right\Vert &\leq &\left(
\sum\limits_{i_{1}=1}^{s_{1}}\alpha _{n,i_{1}}a^{i_{1}}\right) \left(
\sum\limits_{i_{2}=1}^{s_{2}}\beta _{n,i_{2}}^{1}a^{i_{2}}\right) \cdots
\left( \sum\limits_{i_{k-1}=1}^{s_{k-1}}\beta
_{n,i_{k-1}}^{k-2}a^{i_{k-1}}\right) \left\Vert u_{n}^{k-1}-q\right\Vert 
\notag \\
&&+\left\{ \alpha _{n,0}+\left( \sum\limits_{i_{1}=1}^{s_{1}}\alpha
_{n,i_{1}}a^{i_{1}}\right) \beta _{n,0}^{1}+\left(
\sum\limits_{i_{1}=1}^{s_{1}}\alpha _{n,i_{1}}a^{i_{1}}\right) \left(
\sum\limits_{i_{2}=1}^{s_{2}}\beta _{n,i_{2}}^{1}a^{i_{2}}\right) \beta
_{n,0}^{2}\right.  \notag \\
&&+\cdots  \notag \\
&&+\left( \sum\limits_{i_{1}=1}^{s_{1}}\alpha _{n,i_{1}}a^{i_{1}}\right)
\left( \sum\limits_{i_{2}=1}^{s_{2}}\beta _{n,i_{2}}^{1}a^{i_{2}}\right) 
\notag \\
&&\left. \cdots \left( \sum\limits_{i_{k-2}=1}^{s_{k-2}}\beta
_{n,i_{k-2}}^{k-3}a^{i_{k-2}}\right) \beta _{n,0}^{k-2}\right\} \left\Vert
y_{n}-q\right\Vert +\varepsilon _{n}\text{.}  \label{eqn28}
\end{eqnarray}%
Again using (1.4) and Lemma 2, we obtain%
\begin{eqnarray}
\left\Vert u_{n}^{k-1}-q\right\Vert &=&\left\Vert
\sum\limits_{i_{k}=0}^{s_{k}}\beta
_{n,i_{k}}^{k-1}T^{i_{k}}y_{n}-\sum\limits_{i_{k}=0}^{s_{k}}\beta
_{n,i_{k}}^{k-1}T^{i_{k}}q\right\Vert  \notag \\
&\leq &\beta _{n,0}^{k-1}\left\Vert y_{n}-q\right\Vert
+\sum\limits_{i_{k}=1}^{s_{k}}\beta _{n,i_{k}}^{k-1}\left\Vert
T^{i_{k}}y_{n}-T^{i_{k}}q\right\Vert  \notag \\
&\leq &\beta _{n,0}^{k-1}\left\Vert y_{n}-q\right\Vert  \notag \\
&&+\sum\limits_{i_{k}=1}^{s_{k}}\beta _{n,i_{k}}^{k-1}\left\{
\sum\limits_{j=1}^{i_{k}}\binom{i_{k}}{j}a^{i_{k}-j}\varphi ^{j}\left(
\left\Vert q-Tq\right\Vert \right) +a^{i_{k}}\left\Vert y_{n}-q\right\Vert
\right\}  \notag \\
&=&\left( \sum\limits_{i_{k}=0}^{s_{k}}\beta
_{n,i_{k}}^{k-1}a^{i_{k}}\right) \left\Vert y_{n}-q\right\Vert \text{.}
\label{eqn29}
\end{eqnarray}%
Substituting (2.16) in (2.15) we derive%
\begin{eqnarray}
\left\Vert y_{n+1}-q\right\Vert &\leq &\left\{ \alpha _{n,0}+\left(
\sum\limits_{i_{1}=1}^{s_{1}}\alpha _{n,i_{1}}a^{i_{1}}\right) \beta
_{n,0}^{1}+\left( \sum\limits_{i_{1}=1}^{s_{1}}\alpha
_{n,i_{1}}a^{i_{1}}\right) \left( \sum\limits_{i_{2}=1}^{s_{2}}\beta
_{n,i_{2}}^{1}a^{i_{2}}\right) \beta _{n,0}^{2}\right.  \notag \\
&&+\cdots  \notag \\
&&+\left( \sum\limits_{i_{1}=1}^{s_{1}}\alpha _{n,i_{1}}a^{i_{1}}\right)
\left( \sum\limits_{i_{2}=1}^{s_{2}}\beta _{n,i_{2}}^{1}a^{i_{2}}\right)
\cdots \left( \sum\limits_{i_{k-2}=1}^{s_{k-2}}\beta
_{n,i_{k-2}}^{k-3}a^{i_{k-2}}\right) \beta _{n,0}^{k-2}  \notag \\
&&+\left( \sum\limits_{i_{1}=1}^{s_{1}}\alpha _{n,i_{1}}a^{i_{1}}\right)
\left( \sum\limits_{i_{2}=1}^{s_{2}}\beta _{n,i_{2}}^{1}a^{i_{2}}\right) 
\notag \\
&&\left. \cdots \left( \sum\limits_{i_{k-1}=1}^{s_{k-1}}\beta
_{n,i_{k-1}}^{k-2}a^{i_{k-1}}\right) \left(
\sum\limits_{i_{k}=0}^{s_{k}}\beta _{n,i_{k}}^{k-1}a^{i_{k}}\right) \right\}
\left\Vert y_{n}-q\right\Vert +\varepsilon _{n}\text{.}  \label{eqn30}
\end{eqnarray}%
Define%
\begin{eqnarray}
\sigma &:&=\alpha _{n,0}+\left( \sum\limits_{i_{1}=1}^{s_{1}}\alpha
_{n,i_{1}}a^{i_{1}}\right) \beta _{n,0}^{1}+\left(
\sum\limits_{i_{1}=1}^{s_{1}}\alpha _{n,i_{1}}a^{i_{1}}\right) \left(
\sum\limits_{i_{2}=1}^{s_{2}}\beta _{n,i_{2}}^{1}a^{i_{2}}\right) \beta
_{n,0}^{2}  \notag \\
&&+\cdots  \notag \\
&&+\left( \sum\limits_{i_{1}=1}^{s_{1}}\alpha _{n,i_{1}}a^{i_{1}}\right)
\left( \sum\limits_{i_{2}=1}^{s_{2}}\beta _{n,i_{2}}^{1}a^{i_{2}}\right)
\cdots \left( \sum\limits_{i_{k-2}=1}^{s_{k-2}}\beta
_{n,i_{k-2}}^{k-3}a^{i_{k-2}}\right) \beta _{n,0}^{k-2}  \notag \\
&&+\left( \sum\limits_{i_{1}=1}^{s_{1}}\alpha _{n,i_{1}}a^{i_{1}}\right)
\left( \sum\limits_{i_{2}=1}^{s_{2}}\beta _{n,i_{2}}^{1}a^{i_{2}}\right)
\cdots \left( \sum\limits_{i_{k-1}=1}^{s_{k-1}}\beta
_{n,i_{k-1}}^{k-2}a^{i_{k-1}}\right) \left(
\sum\limits_{i_{k}=0}^{s_{k}}\beta _{n,i_{k}}^{k-1}a^{i_{k}}\right) \text{.}
\label{eqn31}
\end{eqnarray}%
We now show that $\sigma \in \left( 0,1\right) $. Since $a^{i_{k}}\in \left[
0,1\right) $, $\alpha _{n,0}>0$, $\sum\limits_{i_{1}=0}^{s_{1}}\alpha
_{n,i_{1}}=1$ and $\sum\limits_{i_{p+1}=0}^{s_{p+1}}\beta _{n,i_{p+1}}^{p}=1$
for $p=\overline{1,k-1}$, we have%
\begin{eqnarray}
\sigma &<&\alpha _{n,0}+\left( 1-\alpha _{n,0}\right) \beta
_{n,0}^{1}+\left( 1-\alpha _{n,0}\right) \left( 1-\beta _{n,0}^{1}\right)
\beta _{n,0}^{2}  \notag \\
&&+\cdots +\left( 1-\alpha _{n,0}\right) \left( 1-\beta _{n,0}^{1}\right)
\cdots \left( 1-\beta _{n,0}^{k-3}\right) \beta _{n,0}^{k-2}  \notag \\
&&+\left( 1-\alpha _{n,0}\right) \left( 1-\beta _{n,0}^{1}\right) \cdots
\left( 1-\beta _{n,0}^{k-2}\right)  \notag \\
&=&1\text{,}  \label{eqn32}
\end{eqnarray}%
that is, $\sigma \in \left( 0,1\right) $. Therefore, an application of Lemma
2 to (2.17) yields $\lim_{n\rightarrow \infty }y_{n}=q$.

Now suppose that $\lim_{n\rightarrow \infty }y_{n}=q$. Then we shall show
that $\lim_{n\rightarrow \infty }\varepsilon _{n}=0$.

Using Lemma 2.2 we have%
\begin{eqnarray}
\varepsilon _{n} &=&\left\Vert y_{n+1}-\alpha
_{n,0}y_{n}-\sum\limits_{i_{1}=1}^{s_{1}}\alpha
_{n,i_{1}}T^{i_{1}}u_{n}^{1}\right\Vert  \notag \\
&\leq &\left\Vert y_{n+1}-q\right\Vert +\left\Vert q-\alpha
_{n,0}y_{n}-\sum\limits_{i_{1}=1}^{s_{1}}\alpha
_{n,i_{1}}T^{i_{1}}u_{n}^{1}\right\Vert  \notag \\
&=&\left\Vert y_{n+1}-q\right\Vert +\left\Vert \alpha _{n,0}\left(
q-y_{n}\right) +\sum\limits_{i_{1}=1}^{s_{1}}\alpha _{n,i_{1}}\left(
T^{i_{1}}q-T^{i_{1}}u_{n}^{1}\right) \right\Vert  \notag \\
&\leq &\left\Vert y_{n+1}-q\right\Vert +\alpha _{n,0}\left\Vert
y_{n}-q\right\Vert +\sum\limits_{i_{1}=1}^{s_{1}}\alpha _{n,i_{1}}\left\Vert
T^{i_{1}}q-T^{i_{1}}u_{n}^{1}\right\Vert  \notag \\
&\leq &\left\Vert y_{n+1}-q\right\Vert +\alpha _{n,0}\left\Vert
y_{n}-q\right\Vert  \notag \\
&&+\sum\limits_{i_{1}=1}^{s_{1}}\alpha _{n,i_{1}}\left\{
\sum\limits_{j=1}^{i_{1}}\binom{i_{1}}{j}a^{i_{1}-j}\varphi ^{j}\left(
\left\Vert q-Tq\right\Vert \right) +a^{i_{1}}\left\Vert
q-u_{n}^{1}\right\Vert \right\}  \notag \\
&\leq &\left\Vert y_{n+1}-q\right\Vert +\alpha _{n,0}\left\Vert
y_{n}-q\right\Vert +\left( \sum\limits_{i_{1}=1}^{s_{1}}\alpha
_{n,i_{1}}a^{i_{1}}\right) \left\Vert q-u_{n}^{1}\right\Vert \text{,}
\label{eqn33}
\end{eqnarray}%
\begin{eqnarray}
\left\Vert q-u_{n}^{1}\right\Vert &=&\left\Vert q-\beta
_{n,0}^{1}y_{n}-\sum\limits_{i_{2}=1}^{s_{2}}\beta
_{n,i_{2}}^{1}T^{i_{2}}u_{n}^{2}\right\Vert  \notag \\
&=&\left\Vert \beta _{n,0}^{1}\left( q-y_{n}\right)
+\sum\limits_{i_{2}=1}^{s_{2}}\beta _{n,i_{2}}^{1}\left(
T^{i_{2}}q-T^{i_{2}}u_{n}^{2}\right) \right\Vert  \notag \\
&\leq &\beta _{n,0}^{1}\left\Vert y_{n}-q\right\Vert
+\sum\limits_{i_{2}=1}^{s_{2}}\beta _{n,i_{2}}^{1}\left\Vert
T^{i_{2}}q-T^{i_{2}}u_{n}^{2}\right\Vert  \notag \\
&\leq &\beta _{n,0}^{1}\left\Vert y_{n}-q\right\Vert  \notag \\
&&+\sum\limits_{i_{2}=1}^{s_{2}}\beta _{n,i_{2}}^{1}\left\{
\sum\limits_{j=1}^{i_{2}}\binom{i_{2}}{j}a^{i_{2}-j}\varphi ^{j}\left(
\left\Vert q-Tq\right\Vert \right) +a^{i_{2}}\left\Vert
q-u_{n}^{2}\right\Vert \right\}  \notag \\
&\leq &\beta _{n,0}^{1}\left\Vert y_{n}-q\right\Vert +\left(
\sum\limits_{i_{2}=1}^{s_{2}}\beta _{n,i_{2}}^{1}a^{i_{2}}\right) \left\Vert
q-u_{n}^{2}\right\Vert \text{,}  \label{eqn34}
\end{eqnarray}%
\begin{eqnarray}
\left\Vert q-u_{n}^{2}\right\Vert &=&\left\Vert q-\beta
_{n,0}^{2}y_{n}-\sum\limits_{i_{3}=1}^{s_{3}}\beta
_{n,i_{3}}^{2}T^{i_{3}}u_{n}^{3}\right\Vert  \notag \\
&=&\left\Vert \beta _{n,0}^{2}\left( q-y_{n}\right)
+\sum\limits_{i_{3}=1}^{s_{3}}\beta _{n,i_{3}}^{2}\left(
T^{i_{3}}q-T^{i_{3}}u_{n}^{3}\right) \right\Vert  \notag \\
&\leq &\beta _{n,0}^{2}\left\Vert y_{n}-q\right\Vert
+\sum\limits_{i_{3}=1}^{s_{3}}\beta _{n,i_{3}}^{2}\left\Vert
T^{i_{3}}q-T^{i_{3}}u_{n}^{3}\right\Vert  \notag \\
&\leq &\beta _{n,0}^{2}\left\Vert y_{n}-q\right\Vert  \notag \\
&&+\sum\limits_{i_{3}=1}^{s_{3}}\beta _{n,i_{3}}^{2}\left\{
\sum\limits_{j=1}^{i_{3}}\binom{i_{3}}{j}a^{i_{3}-j}\varphi ^{j}\left(
\left\Vert q-Tq\right\Vert \right) +a^{i_{3}}\left\Vert
q-u_{n}^{3}\right\Vert \right\}  \notag \\
&\leq &\beta _{n,0}^{2}\left\Vert y_{n}-q\right\Vert +\left(
\sum\limits_{i_{3}=1}^{s_{3}}\beta _{n,i_{3}}^{2}a^{i_{3}}\right) \left\Vert
q-u_{n}^{3}\right\Vert \text{.}  \label{eqn35}
\end{eqnarray}%
It follows from the relation (2.20), (2.21), and (2.22) that%
\begin{eqnarray}
\varepsilon _{n} &\leq &\left\Vert y_{n+1}-q\right\Vert +\alpha
_{n,0}\left\Vert y_{n}-q\right\Vert  \notag \\
&&+\left( \sum\limits_{i_{1}=1}^{s_{1}}\alpha _{n,i_{1}}a^{i_{1}}\right)
\left( \sum\limits_{i_{2}=1}^{s_{2}}\beta _{n,i_{2}}^{1}a^{i_{2}}\right)
\left( \sum\limits_{i_{3}=1}^{s_{3}}\beta _{n,i_{3}}^{2}a^{i_{3}}\right)
\left\Vert q-u_{n}^{3}\right\Vert  \notag \\
&&+\left( \sum\limits_{i_{1}=1}^{s_{1}}\alpha _{n,i_{1}}a^{i_{1}}\right)
\left( \sum\limits_{i_{2}=1}^{s_{2}}\beta _{n,i_{2}}^{1}a^{i_{2}}\right)
\beta _{n,0}^{2}\left\Vert y_{n}-q\right\Vert  \notag \\
&&+\left( \sum\limits_{i_{1}=1}^{s_{1}}\alpha _{n,i_{1}}a^{i_{1}}\right)
\beta _{n,0}^{1}\left\Vert y_{n}-q\right\Vert \text{.}  \label{eqn36}
\end{eqnarray}%
Thus, by induction, we get%
\begin{eqnarray}
\varepsilon _{n} &\leq &\left\Vert y_{n+1}-q\right\Vert +\alpha
_{n,0}\left\Vert y_{n}-q\right\Vert  \notag \\
&&+\left( \sum\limits_{i_{1}=1}^{s_{1}}\alpha _{n,i_{1}}a^{i_{1}}\right)
\left( \sum\limits_{i_{2}=1}^{s_{2}}\beta _{n,i_{2}}^{1}a^{i_{2}}\right)
\cdots \left( \sum\limits_{i_{k-1}=1}^{s_{k-1}}\beta
_{n,i_{k-1}}^{k-2}a^{i_{k-1}}\right) \left\Vert q-u_{n}^{k-1}\right\Vert 
\notag \\
&&+\left( \sum\limits_{i_{1}=1}^{s_{1}}\alpha _{n,i_{1}}a^{i_{1}}\right)
\left( \sum\limits_{i_{2}=1}^{s_{2}}\beta _{n,i_{2}}^{1}a^{i_{2}}\right)
\cdots \left( \sum\limits_{i_{k-2}=1}^{s_{k-2}}\beta
_{n,i_{k-2}}^{k-3}a^{i_{k-2}}\right) \beta _{n,0}^{k-2}\left\Vert
y_{n}-q\right\Vert  \notag \\
&&\cdots  \notag \\
&&+\left( \sum\limits_{i_{1}=1}^{s_{1}}\alpha _{n,i_{1}}a^{i_{1}}\right)
\left( \sum\limits_{i_{2}=1}^{s_{2}}\beta _{n,i_{2}}^{1}a^{i_{2}}\right)
\beta _{n,0}^{2}\left\Vert y_{n}-q\right\Vert  \notag \\
&&+\left( \sum\limits_{i_{1}=1}^{s_{1}}\alpha _{n,i_{1}}a^{i_{1}}\right)
\beta _{n,0}^{1}\left\Vert y_{n}-q\right\Vert \text{.}  \label{eqn37}
\end{eqnarray}%
Utilizing (1.4) and Lemma 2, we obtain%
\begin{eqnarray}
\left\Vert q-u_{n}^{k-1}\right\Vert &=&\left\Vert
\sum\limits_{i_{k}=0}^{s_{k}}\beta
_{n,i_{k}}^{k-1}T^{i_{k}}q-\sum\limits_{i_{k}=0}^{s_{k}}\beta
_{n,i_{k}}^{k-1}T^{i_{k}}y_{n}\right\Vert  \notag \\
&\leq &\beta _{n,0}^{k-1}\left\Vert y_{n}-q\right\Vert
+\sum\limits_{i_{k}=1}^{s_{k}}\beta _{n,i_{k}}^{k-1}\left\Vert
T^{i_{k}}q-T^{i_{k}}y_{n}\right\Vert  \notag \\
&\leq &\beta _{n,0}^{k-1}\left\Vert y_{n}-q\right\Vert  \notag \\
&&+\sum\limits_{i_{k}=1}^{s_{k}}\beta _{n,i_{k}}^{k-1}\left\{
\sum\limits_{j=1}^{i_{k}}\binom{i_{k}}{j}a^{i_{k}-j}\varphi ^{j}\left(
\left\Vert q-Tq\right\Vert \right) +a^{i_{k}}\left\Vert y_{n}-q\right\Vert
\right\}  \notag \\
&=&\left( \sum\limits_{i_{k}=0}^{s_{k}}\beta
_{n,i_{k}}^{k-1}a^{i_{k}}\right) \left\Vert y_{n}-q\right\Vert \text{.}
\label{eqn38}
\end{eqnarray}%
Substituting (2.25) in (2.24) gives%
\begin{eqnarray}
\varepsilon _{n} &\leq &\left\Vert y_{n+1}-q\right\Vert  \notag \\
&&+\left\{ \alpha _{n,0}+\left( \sum\limits_{i_{1}=1}^{s_{1}}\alpha
_{n,i_{1}}a^{i_{1}}\right) \beta _{n,0}^{1}+\left(
\sum\limits_{i_{1}=1}^{s_{1}}\alpha _{n,i_{1}}a^{i_{1}}\right) \left(
\sum\limits_{i_{2}=1}^{s_{2}}\beta _{n,i_{2}}^{1}a^{i_{2}}\right) \beta
_{n,0}^{2}\right.  \notag \\
&&+\cdots  \notag \\
&&+\left( \sum\limits_{i_{1}=1}^{s_{1}}\alpha _{n,i_{1}}a^{i_{1}}\right)
\left( \sum\limits_{i_{2}=1}^{s_{2}}\beta _{n,i_{2}}^{1}a^{i_{2}}\right)
\cdots \left( \sum\limits_{i_{k-2}=1}^{s_{k-2}}\beta
_{n,i_{k-2}}^{k-3}a^{i_{k-2}}\right) \beta _{n,0}^{k-2}  \notag \\
&&+\left( \sum\limits_{i_{1}=1}^{s_{1}}\alpha _{n,i_{1}}a^{i_{1}}\right)
\left( \sum\limits_{i_{2}=1}^{s_{2}}\beta _{n,i_{2}}^{1}a^{i_{2}}\right) 
\notag \\
&&\cdots \left( \sum\limits_{i_{k-1}=1}^{s_{k-1}}\beta
_{n,i_{k-1}}^{k-2}a^{i_{k-1}}\right) \left(
\sum\limits_{i_{k}=0}^{s_{k}}\beta _{n,i_{k}}^{k-1}a^{i_{k}}\right)
\left\Vert y_{n}-q\right\Vert  \label{eqn39}
\end{eqnarray}%
Again define%
\begin{eqnarray}
\sigma &:&=\alpha _{n,0}+\left( \sum\limits_{i_{1}=1}^{s_{1}}\alpha
_{n,i_{1}}a^{i_{1}}\right) \beta _{n,0}^{1}+\left(
\sum\limits_{i_{1}=1}^{s_{1}}\alpha _{n,i_{1}}a^{i_{1}}\right) \left(
\sum\limits_{i_{2}=1}^{s_{2}}\beta _{n,i_{2}}^{1}a^{i_{2}}\right) \beta
_{n,0}^{2}  \notag \\
&&+\cdots  \notag \\
&&+\left( \sum\limits_{i_{1}=1}^{s_{1}}\alpha _{n,i_{1}}a^{i_{1}}\right)
\left( \sum\limits_{i_{2}=1}^{s_{2}}\beta _{n,i_{2}}^{1}a^{i_{2}}\right)
\cdots \left( \sum\limits_{i_{k-2}=1}^{s_{k-2}}\beta
_{n,i_{k-2}}^{k-3}a^{i_{k-2}}\right) \beta _{n,0}^{k-2}  \notag \\
&&+\left( \sum\limits_{i_{1}=1}^{s_{1}}\alpha _{n,i_{1}}a^{i_{1}}\right)
\left( \sum\limits_{i_{2}=1}^{s_{2}}\beta _{n,i_{2}}^{1}a^{i_{2}}\right)
\cdots \left( \sum\limits_{i_{k-1}=1}^{s_{k-1}}\beta
_{n,i_{k-1}}^{k-2}a^{i_{k-1}}\right) \left(
\sum\limits_{i_{k}=0}^{s_{k}}\beta _{n,i_{k}}^{k-1}a^{i_{k}}\right) \text{.}
\label{eqn40}
\end{eqnarray}%
Hence (2.26) becomes%
\begin{equation}
\varepsilon _{n}\leq \left\Vert y_{n+1}-q\right\Vert +\sigma \left\Vert
y_{n}-q\right\Vert \text{.}  \label{eqn41}
\end{equation}%
Using same argument that of first part of the proof we obtain $\sigma \in
\left( 0,1\right) $.

It therefore follows from assumption $\lim_{n\rightarrow \infty }y_{n}=q$
that $\varepsilon _{n}\rightarrow 0$ as $n\rightarrow \infty $.
\end{proof}

\begin{theorem}
Let $x_{0}\in X$ and $\left\{ x_{n}\right\} _{n\in 
%TCIMACRO{\U{2115} }%
%BeginExpansion
\mathbb{N}
%EndExpansion
}$ be a sequence generated by the Kirk-SP iterative scheme (1.3). Suppose
that $T$ has a fixed point $q$. The, the Kirk-SP iterative scheme (1.3) is $%
T $-stable.
\end{theorem}

\begin{proof}
Let $\left\{ y_{n}\right\} _{n\in 
%TCIMACRO{\U{2115} }%
%BeginExpansion
\mathbb{N}
%EndExpansion
}\subset X$, $\varepsilon _{n}=\left\Vert
y_{n+1}-\sum_{i_{1}=0}^{s_{1}}\alpha _{n,i_{1}}T^{i_{1}}u_{n}^{1}\right\Vert 
$, $n=0,1,2,\ldots $, $u_{n}^{1}=\sum_{i_{2}=0}^{s_{2}}\beta
_{n,i_{2}}^{1}T^{i_{2}}u_{n}^{2}$, and $u_{n}^{2}=\sum_{i_{3}=0}^{s_{3}}%
\beta _{n,i_{3}}^{2}T^{i_{3}}y_{n}$. Assume that $\lim_{n\rightarrow \infty
}\varepsilon _{n}=0$. We shall prove that $\lim_{n\rightarrow \infty
}y_{n}=q $.

It follows from (1.3) and Lemma 2 that%
\begin{eqnarray}
\left\Vert y_{n+1}-q\right\Vert &=&\left\Vert
y_{n+1}-\sum\limits_{i_{1}=0}^{s_{1}}\alpha
_{n,i_{1}}T^{i_{1}}u_{n}^{1}+\sum\limits_{i_{1}=0}^{s_{1}}\alpha
_{n,i_{1}}T^{i_{1}}u_{n}^{1}-q\right\Vert  \notag \\
&\leq &\varepsilon _{n}+\left\Vert \sum\limits_{i_{1}=0}^{s_{1}}\alpha
_{n,i_{1}}\left( T^{i_{1}}u_{n}^{1}-T^{i_{1}}q\right) \right\Vert  \notag \\
&\leq &\varepsilon _{n}+\alpha _{n,0}\left\Vert u_{n}^{1}-q\right\Vert
+\sum\limits_{i_{1}=1}^{s_{1}}\alpha _{n,i_{1}}\left\Vert
T^{i_{1}}u_{n}^{1}-T^{i_{1}}q\right\Vert  \notag \\
&\leq &\varepsilon _{n}+\alpha _{n,0}\left\Vert u_{n}^{1}-q\right\Vert 
\notag \\
&&+\sum\limits_{i_{1}=1}^{s_{1}}\alpha _{n,i_{1}}\left\{
\sum\limits_{j=1}^{i_{1}}\binom{i_{1}}{j}a^{i_{1}-j}\varphi ^{j}\left(
\left\Vert q-Tq\right\Vert \right) +a^{i_{1}}\left\Vert
u_{n}^{1}-q\right\Vert \right\}  \notag \\
&=&\varepsilon _{n}+\left( \sum\limits_{i_{1}=0}^{s_{1}}\alpha
_{n,i_{1}}a^{i_{1}}\right) \left\Vert u_{n}^{1}-q\right\Vert \text{,}
\label{eqn42}
\end{eqnarray}%
\begin{eqnarray}
\left\Vert u_{n}^{1}-q\right\Vert &=&\left\Vert
\sum\limits_{i_{2}=0}^{s_{2}}\beta _{n,i_{2}}^{1}\left(
T^{i_{2}}u_{n}^{2}-T^{i_{2}}q\right) \right\Vert  \notag \\
&\leq &\beta _{n,0}^{1}\left\Vert u_{n}^{2}-q\right\Vert
+\sum\limits_{i_{2}=1}^{s_{2}}\beta _{n,i_{2}}^{1}\left\Vert
T^{i_{2}}u_{n}^{2}-T^{i_{2}}q\right\Vert  \notag \\
&\leq &\beta _{n,0}^{1}\left\Vert u_{n}^{2}-q\right\Vert  \notag \\
&&+\sum\limits_{i_{2}=1}^{s_{2}}\beta _{n,i_{2}}^{1}\left\{
\sum\limits_{j=1}^{i_{2}}\binom{i_{2}}{j}a^{i_{2}-j}\varphi ^{j}\left(
\left\Vert q-Tq\right\Vert \right) +a^{i_{2}}\left\Vert
u_{n}^{2}-q\right\Vert \right\}  \notag \\
&=&\left( \sum\limits_{i_{2}=0}^{s_{2}}\beta _{n,i_{2}}^{1}a^{i_{2}}\right)
\left\Vert u_{n}^{2}-q\right\Vert \text{,}  \label{eqn43}
\end{eqnarray}%
and%
\begin{eqnarray}
\left\Vert u_{n}^{2}-q\right\Vert &=&\left\Vert
\sum\limits_{i_{3}=0}^{s_{3}}\beta _{n,i_{3}}^{2}\left(
T^{i_{3}}y_{n}-T^{i_{3}}q\right) \right\Vert  \notag \\
&\leq &\beta _{n,0}^{2}\left\Vert y_{n}-q\right\Vert
+\sum\limits_{i_{3}=1}^{s_{3}}\beta _{n,i_{3}}^{2}\left\Vert
T^{i_{3}}y_{n}-T^{i_{3}}q\right\Vert  \notag \\
&\leq &\beta _{n,0}^{2}\left\Vert y_{n}-q\right\Vert  \notag \\
&&+\sum\limits_{i_{3}=1}^{s_{3}}\beta _{n,i_{3}}^{2}\left\{
\sum\limits_{j=1}^{i_{3}}\binom{i_{3}}{j}a^{i_{3}-j}\varphi ^{j}\left(
\left\Vert q-Tq\right\Vert \right) +a^{i_{3}}\left\Vert y_{n}-q\right\Vert
\right\}  \notag \\
&=&\left( \sum\limits_{i_{3}=0}^{s_{3}}\beta _{n,i_{3}}^{2}a^{i_{3}}\right)
\left\Vert y_{n}-q\right\Vert \text{.}  \label{eqn44}
\end{eqnarray}%
Combining (2.29), (2.30), and (2.31) we get%
\begin{equation}
\left\Vert y_{n+1}-q\right\Vert \leq \varepsilon _{n}+\left(
\sum\limits_{i_{1}=0}^{s_{1}}\alpha _{n,i_{1}}a^{i_{1}}\right) \left(
\sum\limits_{i_{2}=0}^{s_{2}}\beta _{n,i_{2}}^{1}a^{i_{2}}\right) \left(
\sum\limits_{i_{3}=0}^{s_{3}}\beta _{n,i_{3}}^{2}a^{i_{3}}\right) \left\Vert
y_{n}-q\right\Vert \text{.}  \label{eqn45}
\end{equation}%
Define%
\begin{equation}
\sigma :=\left( \sum\limits_{i_{1}=0}^{s_{1}}\alpha
_{n,i_{1}}a^{i_{1}}\right) \left( \sum\limits_{i_{2}=0}^{s_{2}}\beta
_{n,i_{2}}^{1}a^{i_{2}}\right) \left( \sum\limits_{i_{3}=0}^{s_{3}}\beta
_{n,i_{3}}^{2}a^{i_{3}}\right) \text{.}  \label{eqn46}
\end{equation}%
Thus we can rewrite (2.32) as follows%
\begin{equation}
\left\Vert y_{n+1}-q\right\Vert \leq \sigma \left\Vert y_{n}-q\right\Vert
+\varepsilon _{n}\text{.}  \label{eqn47}
\end{equation}%
We now show that $\sigma \in \left( 0,1\right) $. Since $a^{i_{k}}\in \left[
0,1\right) $, $\alpha _{n,0}>0$, $\sum\limits_{i_{1}=0}^{s_{1}}\alpha
_{n,i_{1}}=1$ and $\sum\limits_{i_{p+1}=0}^{s_{p+1}}\beta _{n,i_{p+1}}^{p}=1$
for $p=\overline{1,k-1}$%
\begin{equation}
\sigma <\left( \sum\limits_{i_{1}=0}^{s_{1}}\alpha _{n,i_{1}}\right) \left(
\sum\limits_{i_{2}=0}^{s_{2}}\beta _{n,i_{2}}^{1}\right) \left(
\sum\limits_{i_{3}=0}^{s_{3}}\beta _{n,i_{3}}^{2}\right) =1\text{.}
\label{eqn48}
\end{equation}

Therefore, an application of Lemma 1 to (2.34) yields $\lim_{n\rightarrow
\infty }y_{n}=q$.

Now suppose that $\lim_{n\rightarrow \infty }y_{n}=q$. Then we shall show
that $\lim_{n\rightarrow \infty }\varepsilon _{n}=0$.

Using Lemma 2 we have%
\begin{eqnarray}
\varepsilon _{n} &=&\left\Vert y_{n+1}-\sum\limits_{i_{1}=0}^{s_{1}}\alpha
_{n,i_{1}}T^{i_{1}}u_{n}^{1}\right\Vert  \notag \\
&\leq &\left\Vert y_{n+1}-q\right\Vert +\left\Vert
\sum\limits_{i_{1}=0}^{s_{1}}\alpha _{n,i_{1}}\left(
T^{i_{1}}q-T^{i_{1}}u_{n}^{1}\right) \right\Vert  \notag \\
&\leq &\left\Vert y_{n+1}-q\right\Vert +\alpha _{n,0}\left\Vert
q-u_{n}^{1}\right\Vert +\sum\limits_{i_{1}=1}^{s_{1}}\alpha
_{n,i_{1}}\left\Vert T^{i_{1}}q-T^{i_{1}}u_{n}^{1}\right\Vert  \notag \\
&\leq &\left\Vert y_{n+1}-q\right\Vert +\alpha _{n,0}\left\Vert
q-u_{n}^{1}\right\Vert  \notag \\
&&+\sum\limits_{i_{1}=1}^{s_{1}}\alpha _{n,i_{1}}\left\{
\sum\limits_{j=1}^{i_{1}}\binom{i_{1}}{j}a^{i_{1}-j}\varphi ^{j}\left(
\left\Vert q-Tq\right\Vert \right) +a^{i_{1}}\left\Vert
q-u_{n}^{1}\right\Vert \right\}  \notag \\
&\leq &\left\Vert y_{n+1}-q\right\Vert +\left(
\sum\limits_{i_{1}=0}^{s_{1}}\alpha _{n,i_{1}}a^{i_{1}}\right) \left\Vert
q-u_{n}^{1}\right\Vert \text{,}  \label{eqn49}
\end{eqnarray}%
\begin{eqnarray}
\left\Vert q-u_{n}^{1}\right\Vert &=&\left\Vert
\sum\limits_{i_{2}=0}^{s_{2}}\beta _{n,i_{2}}^{1}\left(
T^{i_{2}}q-T^{i_{2}}u_{n}^{2}\right) \right\Vert  \notag \\
&\leq &\beta _{n,0}^{1}\left\Vert q-u_{n}^{2}\right\Vert
+\sum\limits_{i_{2}=1}^{s_{2}}\beta _{n,i_{2}}^{1}\left\Vert
T^{i_{2}}q-T^{i_{2}}u_{n}^{2}\right\Vert  \notag \\
&\leq &\beta _{n,0}^{1}\left\Vert q-u_{n}^{2}\right\Vert  \notag \\
&&+\sum\limits_{i_{2}=1}^{s_{2}}\beta _{n,i_{2}}^{1}\left\{
\sum\limits_{j=1}^{i_{2}}\binom{i_{2}}{j}a^{i_{2}-j}\varphi ^{j}\left(
\left\Vert q-Tq\right\Vert \right) +a^{i_{2}}\left\Vert
q-u_{n}^{2}\right\Vert \right\}  \notag \\
&\leq &\left( \sum\limits_{i_{2}=0}^{s_{2}}\beta
_{n,i_{2}}^{1}a^{i_{2}}\right) \left\Vert q-u_{n}^{2}\right\Vert \text{,}
\label{eqn50}
\end{eqnarray}%
and%
\begin{eqnarray}
\left\Vert q-u_{n}^{2}\right\Vert &=&\left\Vert
\sum\limits_{i_{3}=0}^{s_{3}}\beta _{n,i_{3}}^{2}\left(
T^{i_{3}}q-T^{i_{3}}y_{n}\right) \right\Vert  \notag \\
&\leq &\beta _{n,0}^{2}\left\Vert q-y_{n}\right\Vert
+\sum\limits_{i_{3}=1}^{s_{3}}\beta _{n,i_{3}}^{2}\left\Vert
T^{i_{3}}q-T^{i_{3}}y_{n}\right\Vert  \notag \\
&\leq &\beta _{n,0}^{2}\left\Vert q-y_{n}\right\Vert  \notag \\
&&+\sum\limits_{i_{3}=1}^{s_{3}}\beta _{n,i_{3}}^{2}\left\{
\sum\limits_{j=1}^{i_{3}}\binom{i_{3}}{j}a^{i_{3}-j}\varphi ^{j}\left(
\left\Vert q-Tq\right\Vert \right) +a^{i_{3}}\left\Vert q-y_{n}\right\Vert
\right\}  \notag \\
&=&\left( \sum\limits_{i_{3}=0}^{s_{3}}\beta _{n,i_{3}}^{2}a^{i_{3}}\right)
\left\Vert y_{n}-q\right\Vert \text{.}  \label{eqn51}
\end{eqnarray}%
It follow from (2.36), (2.37), and (2.38) that%
\begin{equation}
\varepsilon _{n}\leq \left\Vert y_{n+1}-q\right\Vert +\left(
\sum\limits_{i_{1}=0}^{s_{1}}\alpha _{n,i_{1}}a^{i_{1}}\right) \left(
\sum\limits_{i_{2}=0}^{s_{2}}\beta _{n,i_{2}}^{1}a^{i_{2}}\right) \left(
\sum\limits_{i_{3}=0}^{s_{3}}\beta _{n,i_{3}}^{2}a^{i_{3}}\right) \left\Vert
y_{n}-q\right\Vert \text{.}  \label{eqn52}
\end{equation}%
Again define%
\begin{equation}
\sigma :=\left( \sum\limits_{i_{1}=0}^{s_{1}}\alpha
_{n,i_{1}}a^{i_{1}}\right) \left( \sum\limits_{i_{2}=0}^{s_{2}}\beta
_{n,i_{2}}^{1}a^{i_{2}}\right) \left( \sum\limits_{i_{3}=0}^{s_{3}}\beta
_{n,i_{3}}^{2}a^{i_{3}}\right) \text{.}  \label{eqn53}
\end{equation}%
Using same argument as that of first part of the proof we obtain $\sigma \in
\left( 0,1\right) $ and it thus follows from assumption $\lim_{n\rightarrow
\infty }y_{n}=q$ that $\varepsilon _{n}\rightarrow 0$ as $n\rightarrow
\infty $.
\end{proof}

\begin{remark}
Theorem 1 is a generalization and extension of both Theorem 1 and Theorem 2
of Berinde \cite{Berinde}, Theorem 2 and Theorem 3 of Kannan \cite{Kannan},
Theorem 3 of Kannan \cite{Kannan1}, Theorem 4 of Rhoades \cite{Rhds1},
Theorem 8 of Rhoades \cite{Rhds2}, Theorem 2.1 of Olatinwo \cite{Olatinwo0},
Theorem 2.6 of Hussain et al \cite{Hussain et al}, and Theorem 3.1 of \c{S}%
oltuz and Grosan \cite{Data Is 2}. Theorems 2 is a generalization and
extension of Theorem 2 of Osilike \cite{Osilike1}, Theorem 2 and Theorem 5
of Osilike and Udomene \cite{Osilike} as well as Theorem 3 of Olatinwo et
al. \cite{Olantiwo4}, Theorem 3.1 and Theorem 3.2 of Olatinwo \cite{Olatinwo}
and Theorem 3.1 of Chugh and Kumar \cite{CR}.
\end{remark}

\begin{acknowledgement}
The first two authors would like to thank Y\i ld\i z Technical University
Scientific Research Projects Coordination Unit under project number BAPK
2012-07-03-DOP02 for financial support during the preparation of this
manuscript.
\end{acknowledgement}


\begin{thebibliography}{99}
\bibitem{FVB} G\"{u}rsoy, F, Karakaya, V, Rhoades, BE: The equivalence
between a new multistep iteration, S-iteration and some other iterative
schemes. (Communicated).
\end{thebibliography}
\end{document}